\documentclass[11pt]{amsart}
\usepackage[latin1]{inputenc}
\usepackage[T1]{fontenc}
\usepackage{hyperref}
\usepackage{amsfonts,dsfont}
\usepackage{amsmath}
\usepackage{epsf, subfigure, verbatim}
\usepackage{amssymb}
\usepackage{amsthm}
\usepackage{latexsym}
\usepackage{layout}
\usepackage{amsrefs}
\usepackage{color}

\usepackage[titletoc]{appendix}

\newtheorem{theorem}{Theorem}

\newtheorem{lemma}{Lemma}[section]
\newtheorem{corollary}{Corollary}

\newtheorem{remark}{Remark}[section]

\newcommand{\ee}{\mathrm{e}} 
\newcommand{\exptime}{\mathbf{e}_p} 
\newcommand{\reals}{\mathbb{R}}
\newcommand{\ind}{\mathds{1}}
\newcommand{\e}{\mathbb{E}}
\newcommand{\p}{\mathbb{P}}
\newcommand{\realsift}{\mathtt c}
\newcommand{\scale}{W}
\newcommand{\qscale}{W^{(q)}}
\newcommand{\pscale}{W^{(p)}}
\newcommand{\pqscale}{W^{(p+q)}}

\newcommand{\zscale}{Z}
\newcommand{\zqscale}{Z^{(q)}}
\newcommand{\zpscale}{Z^{(p)}}

\newcommand{\wapq}{\mathcal W_a^{(p,q)}}

\newcommand{\bP}{\mathbb{P}}

\newcommand{\cH}{\mathcal{H}}
\newcommand{\dR}{\mathbb{R}}
\newcommand{\dE}{\mathbb{E}}
\newcommand{\dP}{\mathbb{P}}
\newcommand{\U}[3]{{U^{(#1)}(#2,\mathrm{d}y;#3)}} 
\newcommand{\Up}[3]{{U_+^{(#1)}(#2,\mathrm{d}y;#3)}} 
\newcommand{\Um}[3]{{U_-^{(#1)}(#2,\mathrm{d}y;#3)}} 
\newcommand{\dy}{\mathrm{d}y}
\newcommand{\dt}{\mathrm{d}t}

\newcommand{\PAR}[1]{{\left(#1\right)}} 
\newcommand{\SBRA}[1]{{\left[#1\right]}} 
\newcommand{\BRA}[1]{{\left\{#1\right\}}} 

\begin{document}

\title[]{Joint distribution of a spectrally negative L\'evy process and its occupation time, with step option pricing in view}

\author[Gu\'erin \& Renaud]{H\'el\`ene Gu\'erin}
\address{IRMAR, Universit\'e de Rennes 1, Campus de Beaulieu
35042 Rennes Cedex France}
\email{helene.guerin@univ-rennes1.fr}

\author[]{Jean-Fran\c{c}ois Renaud}
\address{D\'epartement de math\'ematiques, Universit\'e du Qu\'ebec \`a Montr\'eal (UQAM), 201 av.\ Pr\'esident-Kennedy, Montr\'eal (Qu\'ebec) H2X 3Y7, Canada}
\email{renaud.jf@uqam.ca}

\date{\today}


\begin{abstract}
For a spectrally negative L\'evy process $X$, we study the following distribution:
$$
\e_x \left[ \mathrm{e}^{- q \int_0^t \ind_{(a,b)} (X_s) \mathrm{d}s } ; X_t \in \mathrm{d}y \right] ,
$$
where $-\infty \leq a < b < \infty$, and where $q,t>0$ and $x \in \reals$. More precisely, we identify the Laplace transform with respect to $t$ of this measure in terms of the scale functions of the underlying process. 
Our results are then used to price step options and the particular case of an exponential spectrally negative L\'evy jump-diffusion model is discussed.
\end{abstract}

\maketitle

\textit{Keywords:} Occupations times; Spectrally negative Lévy processes; Fluctuation theory; Scale functions; Step options


\tableofcontents



\section{Introduction}

One of Paul L\'evy's arcsine laws gives the distribution of the occupation time of the positive/negative half-line for a standard Brownian motion. More precisely, if $\BRA{B_t,t\geq 0}$ is a standard Brownian motion, then, for $t>0$,
$$
\p \left( \int_0^t \ind_{(-\infty,0)} (B_u) \leq s \right) = \frac{2}{\pi} \arcsin \left( \sqrt{\frac{s}{t}} \right) \ind_{(0,t)}(s) .
$$
This result was then extended to a Brownian motion with drift by Akahori \cite{akahori1995} and Tak\'acs \cite{takacs1996}.

In the last few years, several papers have looked at the distribution of functionals involving occupation times of a spectrally negative L\'evy process (SNLP), in each case over an infinite time horizon. First, in \cite{landriaultetal2011}, the Laplace transform of the occupation time of semi-infinite intervals for a SNLP has been derived. More precisely, the Laplace transform of
$$
\int_0^{\infty} \ind_{(-\infty,0)} (X_s) \mathrm{d}s \quad \text{and} \quad \int_0^{\tau_{-b}^-} \ind_{(-\infty,0)} (X_s) \mathrm{d}s ,
$$
where $X = \BRA{X_t,t \geq 0}$ is a spectrally negative L\'evy process and where $\tau_{-b}^- = \inf \{t > 0 \colon X_t < -b \}$ with $b>0$, were expressed in terms of the Laplace exponent and the scale functions of the underlying process $X$. Then, in \cite{loeffenetal2014}, those results were significantly extended, first by considering more general quantities, i.e.
$$
\left( \tau_0^- , \int_0^{\tau_0^-} \ind_{(a,b)} (X_s) \mathrm{d}s \right) \quad \text{and} \quad \left( \tau_c^+ , \int_0^{\tau_c^+} \ind_{(a,b)} (X_s) \mathrm{d}s \right) ,
$$
where
$$
\tau_0^- = \inf \{t > 0 \colon X_t < 0 \} \quad \text{and} \quad \tau_c^+ = \inf \{t > 0 \colon X_t > c \} ,
$$
and where $0\leq a \leq b\leq c$, and by obtaining considerably simpler expressions for the joint Laplace transforms. Note that \cite{kyprianouetal2012} and \cite{renaud_2014} have also looked at the abovementioned quantities involving occupation times, but for a so-called refracted L\'evy process, while similar quantities for diffusion processes were studied in \cite{li_zhou_2013} and \cite{lachal_2013}.

In this paper, we are interested in the joint distribution of a spectrally negative L\'evy process and its occupation time when both are sampled at a fixed time. This is closer in spirit to L\'evy's arcsine law and also much more useful for financial applications, especially for the pricing of occupation time options.

\subsection{Occupation time options}\label{sec:intro-options}


Let the risk-neutral price of an asset $S = \{ S_t , t \geq 0\}$ be of the form:
$$
S_t = S_0 \mathrm{e}^{X_t} ,
$$
where $X = \{ X_t , t \geq 0\}$ is the log-return process. For example, in the Black-Scholes-Merton model, $X$ is a Brownian motion with drift. The time spent by $S$ in an interval $I$, or equivalently the time spent by $X$ in an interval $I^\prime$, from time $0$ to time $T$, is given by
$$
 \int_0^T \ind_{\{S_t \in I\}} \mathrm{d}t = \int_0^T \ind_{\{X_t \in I^\prime\}} \mathrm{d}t .
$$

Introduced by Linetsky \cite{linetsky_1999}, (barrier) step options are exotic options linked to occupation times of the underlying asset price process. They are generalized barrier options: instead of being activated (or canceled) when the underlying asset price crosses a barrier, which is a problem from a risk management point of view, the payoff of occupation-time options will depend on the time spent above/below this barrier. Therefore, the change of value occurs more gradually. For instance, a (down-and-out call) step option admits the following payoff:
$$
\mathrm{e}^{- \rho \int_0^T \ind_{\{S_t \leq L\}} \mathrm{d}t } \left( S_T - K \right)_+ = \mathrm{e}^{- \rho \int_0^T \ind_{\{X_t \leq \ln(L/S_0)\}} \mathrm{d}t } \left( S_0 \mathrm{e}^{X_T} - K \right)_+ ,
$$
where $\rho > 0$ is called the \textit{knock-out rate}.
%
Therefore, its price can be written as
\begin{multline*}
C(T) :=\mathrm{e}^{-r T} \e \left[ \mathrm{e}^{- \rho \int_0^T \ind_{\{S_t \leq L\}} \mathrm{d}t } \left( S_T - K \right)_+ \right] \\
= \mathrm{e}^{-rT} \int_{\ln \left( K/S_0 \right)}^\infty \left( S_0 \mathrm{e}^{y} - K \right) \e \left[ \mathrm{e}^{- \rho \int_0^T \ind_\BRA{X_s\leq \ln(L/S_0)} \mathrm{d}s } ; X_T \in \mathrm{d}y \right]
\end{multline*}
with $r$ the risk-free interest rate and its Laplace transform, with respect to the time of maturity $T$, can be written
\begin{multline*}
\int_0^\infty \mathrm{e}^{-pT} C(T) \mathrm{d}T \\
= \int_0^\infty \mathrm{e}^{-(p+r)T} \int_{\ln \left( K/S_0 \right)}^\infty \left( S_0 \mathrm{e}^{y} - K \right) \e \left[ \mathrm{e}^{- \rho \int_0^T \ind_\BRA{X_s\leq \ln(L/S_0)} \mathrm{d}s } ; X_T \in \mathrm{d}y \right] \mathrm{d}T \\
=\int_{\ln \left( K/S_0 \right)}^\infty \left( S_0 \mathrm{e}^{y} - K \right) \int_0^\infty \mathrm{e}^{-(p+r)T} \e \left[ \mathrm{e}^{- \rho \int_0^T \ind_\BRA{X_s\leq \ln(L/S_0)} \mathrm{d}s } ; X_T \in \mathrm{d}y \right] \mathrm{d}T .
\end{multline*}
Thus, pricing step options boils down to identyfing this distribution:
$$
\e \left[ \mathrm{e}^{- \rho \int_0^T \ind_{(-\infty,b)} (X_s) \mathrm{d}s } ; X_T \in \mathrm{d}y \right] ,
$$
for a given value $b \in \mathbb{R}$. Other occupation time options can also be priced using this distribution. For references on occupation-time option pricing, see e.g.\ \cites{cai_chen_wan_2010, hugonnier_1999}.\\

The rest of the paper is organized as follows. In Section 2, we give the necessary background on spectrally negative L\'evy processes and their scale functions. In Section 3, the main results are presented, while their proofs are left for the Appendix. Finally, in section 4, we consider the pricing step options question and a specific example of a jump-diffusion process with hyper-exponential jumps.

\section{Spectrally negative L\'evy processes}

On the filtered probability space $(\Omega, \mathcal{F}, (\mathcal F_t)_{t\geq0},\mathbb P)$, let $X = \BRA{X_t,t \geq 0}$ be a spectrally negative L\'evy process (SNLP), that is a  process with stationary and independent increments and no positive jumps. Hereby we exclude the case of $X$ having monotone paths. As the L\'{e}vy process $X$ has no positive jumps, its Laplace transform exists: for $\lambda,t \geq 0$,
$$
\e \left[ \mathrm{e}^{\lambda X_t} \right] = \mathrm{e}^{t \psi(\lambda)} ,
$$
where
$$
\psi(\lambda) = \gamma \lambda + \frac{1}{2} \sigma^2 \lambda^2 + \int^{\infty}_0 \left( \mathrm{e}^{-\lambda z} - 1 + \lambda z \ind_{(0,1]}(z) \right) \Pi(\mathrm{d}z) ,
$$
for $\gamma \in \reals$ and $\sigma \geq 0$, and where $\Pi$ is a $\sigma$-finite measure on $(0,\infty)$ such that
$$
\int^{\infty}_0 (1 \wedge z^2) \Pi(\mathrm{d}z) < \infty .
$$
This measure is called the L\'{e}vy measure of $X$, while $(\gamma,\sigma,\Pi)$ is refered to as the L\'evy triplet of $X$. Note that for convenience we define the L\'evy measure in such a way that it is a measure on the positive half line instead of the negative half line. Further, note that $\e \left[ X_1 \right] = \psi'(0+)$.

There exists a function $\Phi \colon [0,\infty) \to [0,\infty)$ defined by $\Phi(q) = \sup \{ \lambda \geq 0 \mid \psi(\lambda) = q\}$ (the right-inverse of $\psi$) such that
$$
\psi ( \Phi(q) ) = q, \quad q \geq 0 .
$$
We have that $\Phi(q)=0$ if and only if $q=0$ and $\psi'(0+)\geq0$.

The process $X$ has jumps of bounded variation if $\int^{1}_0 z \Pi(\mathrm{d}z)<\infty$. In that case we denote by
$\realsift:=\gamma+\int^{1}_0 z \Pi(\mathrm{d}z) > 0$ the so-called drift of $X$ which can now be written as
$$
X_t = \realsift t - S_t+\sigma B_t ,
$$
where $S = \BRA{S_t,t \geq 0}$ is a driftless subordinator (for example a Gamma process or a compound Poisson process with positive jumps).

For more details on spectrally negative L\'{e}vy processes, the reader is referred to \cites{kyprianou2006}. 

\subsection{Scale functions and fluctuation identities} \label{sec:scaleF}

We now recall the definition of the $q$-scale function $W^{(q)}$. For $q \geq 0$, the $q$-scale function of the process $X$ is defined as the continuous function on $[0,\infty)$ with Laplace transform
\begin{equation}\label{def_scale}
\int_0^{\infty} \mathrm{e}^{- \lambda y} W^{(q)} (y) \mathrm{d}y = \frac{1}{\psi(\lambda) - q} , \quad \text{for $\lambda > \Phi(q)$.}
\end{equation}
This function is unique, positive and strictly increasing for $x\geq0$ and is further continuous for $q\geq0$. We extend $W^{(q)}$ to the whole real line by setting $W^{(q)}(x)=0$ for $x<0$.  We write $W = W^{(0)}$ when $q=0$. The initial value of $W^{(q)}$ is known to be
\begin{equation*}
W^{(q)}(0)=
\begin{cases}
1/\realsift & \text{when $\sigma=0$ and $\int_{0}^1 z \Pi(\mathrm{d}z) < \infty$},  \\
0 & \text{otherwise},
\end{cases}
\end{equation*}
where we used the following definition: $W^{(q)}(0) = \lim_{x \downarrow 0} W^{(q)}(x)$. We also have that, when $\psi'(0+) > 0$,
$$
\lim_{x \to \infty} W(x) = \frac{1}{\psi'(0+)}.
$$
Finally, we recall the useful relation (see Equation (6) in \cite{loeffenetal2014})
\begin{equation}\label{eq:sym} 
(q-p)\int_0^x\pscale(x-y)\qscale(y)\mathrm{d}y=\qscale(x)-\pscale(x).
\end{equation}

We will also frequently use the function
\begin{equation}\label{eq:zqscale}
Z^{(q)}(x) = 1 + q \int_0^x{W}^{(q)}(y)\mathrm dy, \quad x\in\mathbb R.
\end{equation}

Let us introduce some notations. The law of $X$  such that $X_0 = x$ is denoted by $\p_x$ and the corresponding expectation by $\e_x$. We write $\p$ and $\e$ when $x=0$. Finally, for a random variable $Z$ and an event $A$, $\e [Z;A] := \e [Z \ind_A]$.

Now, for any $a,b \in \reals$, define the stopping times
$$
\tau_a^- = \inf \{t > 0 \colon X_t < a \} \quad \text{and} \quad \tau_b^+ = \inf \{t > 0 \colon X_t > b \} ,
$$
with the convention $\inf\emptyset=\infty$.

It is well known that, if $a \leq x \leq c$, then the solution to the two-sided exit problem for $X$ is given by
\begin{equation}\label{E:exitabove}
\e_x \left[ \ee^{-q \tau_c^+} ; \tau_c^+ < \tau_a^- \right] = \frac{W^{(q)}(x-a)}{W^{(q)}(c-a)} ,
\end{equation}
\begin{equation}\label{E:exitbelow}
\e_x \left[ \ee^{-q \tau_a^-} ; \tau_a^- < \tau_c^+ \right] = Z^{(q)}(x-a) - \frac{Z^{(q)}(c-a)}{W^{(q)}(c-a)} W^{(q)}(x-a) .
\end{equation}

It is well known (see e.g.\ \cite{kyprianou2006}) that under the change of measure given by
$$
\left. \frac{d\mathbb{P}_{x}^{c}}{d\mathbb{P}_{x}} \right|_{\mathcal{F}_{t}} = \exp \left\lbrace c (X_t-x) -\psi(c) t \right\rbrace ,
$$
for $c$ such that $\mathbb{E}_x \left[ \mathrm{e}^{c X_1} \right] < \infty$, $X$ is a spectrally negative L\'evy process with Laplace exponent (under $\mathbb{P}_{x}^c$)
\begin{equation}\label{eq:psic}
\psi_c (\theta) = \psi (\theta+c) - \psi (c) ,
\end{equation}
for $\theta \geq - c$. Its right-inverse function is then given by
$$
\Phi_c (q) = \Phi(q+\psi(c)) - c .
$$
We write $W^{(q)}_c$ and $Z^{(q)}_c$ for the corresponding scale functions. Note that Lemma 8.4 in \cite{kyprianou2006} states that, for $x\geq 0$,
$$
\qscale (x) = \mathrm{e}^{c x} W^{(q - \psi(c))}_c (x) .
$$
A direct application then gives us, for $p,q\geq 0$,
\begin{equation}\label{E:scale_measure_change}
\pqscale(x) = \mathrm{e}^{\Phi(p)x}W_{\Phi(p)}^\PAR{q}(x) .
\end{equation}
Note also that in this case, $\psi'_{\Phi(p)} (0+) = \psi' \left( \Phi(p) \right) > 0$ and $\Phi_{\Phi(p)} (0) = 0$ if $p>0$.

For examples and numerical techniques related to the computation of scale functions, we suggest to look at \cite{kuznetsovetal2012}.

\subsection{The $p$-potential measure}

For $x\in[0,a]$ with $a>0$, 
$$
\U{p}{x}{a}=\int_0^{+\infty} \mathrm{e}^{-pt}\bP_x\PAR{X_t\in \dy ;t<\tau_0^-\wedge \tau_a^+}\dt
$$ 
is the $p$-potential measure of $X$ killed on exiting $[0,a]$ (see \cite[Theorem 8.7]{kyprianou2006} ) and it has a density supported on $[0,a]$ given by
\begin{equation}\label{eq:U}
\U{p}{x}{a}=\BRA{{\pscale(x)\pscale(a-y)\over \pscale(a)}-\pscale(x-y)}\dy . 
\end{equation}

\bigskip 
For $a\in \reals$ and $x\leq a$, we denote by 
\[
\Up{p}{x}{a}=\int_0^{+\infty} \mathrm{e}^{-pt}\bP_x\PAR{X_t\in \dy ;t< \tau_a^+}\dt ,
\]
the $p$-potential measure of $X$ killed on exiting $(-\infty,a]$.
We notice that for $m>0$ large enough
\begin{multline*}
\int_0^{+\infty} \mathrm{e}^{-pt}\bP_{x}\PAR{X_t\in \dy ;t<\tau_{-m}^-\wedge \tau_a^+}\dt \\
=\int_0^{+\infty} \mathrm{e}^{-pt}\bP_{x+m}\PAR{X_t\in \dy +m;t<\tau_0^-\wedge \tau_{a+m}^+}\dt \\
=\BRA{{\pscale(x+m)\pscale(a-y)\over \pscale(a+m)}-\pscale(x-y)}\dy .
\end{multline*}
From 
${\pscale(x+m)\over \pscale(a+m)} = \mathrm{e}^{\Phi(p)(x-a)}{W_{\Phi(p)}(x+m)\over W_{\Phi(p)}(a+m)},$
taking $m\rightarrow +\infty$ in the last equality, we obtain, for $x,y\leq a$,
\begin{equation}\label{eq:Up}
\Up{p}{x}{a}=\BRA{ \mathrm{e}^{\Phi(p)(x-a)}\pscale(a-y)-\pscale(x-y)}\dy .
\end{equation}

\bigskip
For $a\in\reals$ and $x\geq a$, we now introduce 
\[
\Um{p}{x}{a}=\int_0^{+\infty} \mathrm{e}^{-pt}\bP_x\PAR{X_t\in \dy ;t< \tau_a^-}\dt ,
\]
the $p$-potential measure of $X$ killed on exiting $[a,+\infty)$.
From \cite[Corollary 8.8]{kyprianou2006},  we know that
$$\Um{p}{x}{0}=\BRA{ \mathrm{e}^{-\Phi(p)y}\pscale(x)-\pscale(x-y)}\dy .$$
Then we deduce that, for $x,y\geq a$,
\begin{multline*}
\Um{p}{x}{a} = \int_0^{+\infty} \mathrm{e}^{-pt}\bP_{x-a}\PAR{X_t\in \dy -a;t< \tau_0^-}\dt \\
=\BRA{ \mathrm{e}^{-\Phi(p)(y-a)}\pscale(x-a)-\pscale(x-y)}\dy .
\end{multline*}
We extend the definition of $\Um{p}{x}{a}$ for $x\in\dR$ by
\begin{equation}\label{eq:Um}
\Um{p}{x}{a}
=\BRA{ \mathrm{e}^{-\Phi(p)(y-a)}\pscale(x-a)-\pscale(x-y)}\ind_{x,y\geq a}\mathrm{d}y.
\end{equation}

\section{Main results}

We are interested in the following distribution: for fixed $q,t>0$, and for all $x \in \reals$,
\begin{equation}\label{E:qty_of_interest}
\e_x \left[ \mathrm{e}^{- q \int_0^t \ind_{(0,a)} (X_s) \mathrm{d}s } ; X_t \in \mathrm{d}y \right] .
\end{equation}
We will consider the Laplace transform, with respect to $t$, of the expectation in~\eqref{E:qty_of_interest}: for all $x \in \reals$, set for $p>0$
\begin{align*}
v(x,\mathrm{d}y ) &:= \int_0^\infty  \mathrm{e}^{-p t} \e_x \left[ \mathrm{e}^{- q \int_0^{t} \ind_{(0,a)} (X_s) \mathrm{d}s } ; X_t \in \mathrm{d}y \right] \mathrm{d}t .
\end{align*}
We notice that it can be written
\begin{equation}\label{eq:defv}
v(x,\mathrm{d}y )={1\over p} \e_x \left[ \mathrm{e}^{- q \int_0^{\exptime} \ind_{(0,a)} (X_s) \mathrm{d}s } ; X_{\exptime} \in \mathrm{d}y \right] ,
\end{equation}
where $\exptime$ is an exponentially distributed random variable (independent of $X$) with mean $1/p$.

\begin{remark}
In what follows, for sake of simplicity, we will omit indicator functions of the form $\ind_{\tau < \infty}$ in our expectations, where $\tau$ is a first-passage stopping time when there is no confusion. For example, we will write $\e_x \left[ \ee^{-q \tau_a^-} \right]$ instead of $\e_x \left[ \ee^{-q \tau_a^-} ; \tau_a^- < \infty \right]$.
\end{remark}

\subsection{A probabilistic decomposition}

Recall that if we assume $X$ has paths of bounded variation (BV), then, for any $b\in\dR$, we have $X_{\tau_b^-}<b$ almost surely; in other words, $X$ does not creep downward. Also, in that case, $\pscale(0) > 0$. The following proof of lemma is based on representation~\eqref{eq:defv} of $v(.,\mathrm{d}y)$. We discuss according to the position of the hitting times of $(0,a)$ with respect to the exponential time $\exptime$.
\begin{lemma} \label{lem:vgeneral}
Assume $a>0$. For a spectrally negative L\'evy process  $X$ with paths of BV, $v(.,\mathrm{d}y )$ satisfies, for all $x\in\dR$,
\begin{multline*}
v(x,\mathrm{d}y )=\Um{p}{x}{a}\ind_{y>a}+\dE_x\SBRA{ \mathrm{e}^{-p\tau_a^-}\U{p+q}{X_{\tau_a^-}}{a}}\ind_{y\in[0,a]}\\
+{v(a,\mathrm{d}y )\over \pqscale(a)}\dE_x\SBRA{ \mathrm{e}^{-p\tau_a^-}\pqscale(X_{\tau_a^-})}+v(0,\mathrm{d}y )\dE_x\SBRA{ \mathrm{e}^{-p\tau_a^-}\dE_{X_{\tau_a^-}}\SBRA{ \mathrm{e}^{-(p+q)\tau_0^-+\Phi(p)X_{\tau_0^-}};\tau_0^-< \tau_a^+}} \\
+\dE_x\SBRA{ \mathrm{e}^{-p\tau_a^-}\dE_{X_{\tau_a^-}}\SBRA{ \mathrm{e}^{-(p+q)\tau_0^-}\Up{p}{X_{\tau_0^-}}{0};\tau_0^-< \tau_a^+}}\ind_{y<0}.
\end{multline*}
\end{lemma}
\begin{proof}We use representation \eqref{eq:defv} of $v(.,\dy)$.
For $x<0$, we notice that
\begin{align}
 v(x,\mathrm{d}y )&={1\over p}\bP_x\PAR{X_{\exptime}\in \mathrm{d}y;\exptime<\tau_0^+ }+v(0,\mathrm{d}y )\bP_x(\tau_0^+<\exptime)\notag\\
&=\Up{p}{x}{0}\ind_{y<0}+ \mathrm{e}^{\Phi(p)x}v(0,\mathrm{d}y ).\label{eq:vxneg}
\end{align}
Let us now consider $0\leq x<a$ and we have
\begin{align*}
v(x,\mathrm{d}y )&={1\over p}\dE_x\SBRA{ \mathrm{e}^{-q\exptime};X_{\exptime}\in \mathrm{d}y ;\exptime<\tau_0^-\wedge \tau_a^+}+v(a,\mathrm{d}y )\dE_x\SBRA{ \mathrm{e}^{-q\tau_a^+}; \tau_a^+<\tau_0^-\wedge \exptime }\\
&\hskip 1cm +\dE_x\SBRA{ \mathrm{e}^{-q\tau_0^-}v(X_{\tau_0^-},\mathrm{d}y );\tau_0^-<\exptime \wedge \tau_a^+}\\
&=\U{p+q}{x}{a}\ind_{y\in[0,a]}+v(a,\mathrm{d}y )\dE_x\SBRA{ \mathrm{e}^{-(p+q)\tau_a^+};\tau_a^+<\tau_0^-}\\
&\hskip 1 cm +\dE_x\SBRA{ \mathrm{e}^{-(p+q)\tau_0^-}v(X_{\tau_0^-},\mathrm{d}y );\tau_0^-< \tau_a^+}\\
&=\U{p+q}{x}{a}\ind_{y\in[0,a]}+v(a,\mathrm{d}y ){\pqscale(x)\over\pqscale(a)}
+\dE_x\SBRA{ \mathrm{e}^{-(p+q)\tau_0^-}v(X_{\tau_0^-},\mathrm{d}y );\tau_0^-< \tau_a^+}
\end{align*}
Since $X$ is of bounded variation, $X_{\tau_0^-}<0$ a.s.. Consequently, using \eqref{eq:vxneg}, we notice that for any $x<a$,
\begin{multline}\label{eq:va}
v(x,\mathrm{d}y )=\U{p+q}{x}{a}\ind_{y\in[0,a]}+v(a,\mathrm{d}y ){\pqscale(x)\over\pqscale(a)}
\\
+v(0,\mathrm{d}y )\dE_x\SBRA{ \mathrm{e}^{-(p+q)\tau_0^-+\Phi(p)X_{\tau_0^-}};\tau_0^-< \tau_a^+} +\dE_x\SBRA{ \mathrm{e}^{-(p+q)\tau_0^-}\Up{p}{X_{\tau_0^-}}{0};\tau_0^-< \tau_a^+}\ind_{y<0}.
\end{multline}
We now study the last case: let $x\geq a$,
\begin{align*}
v(x,\mathrm{d}y )&={1\over p}\bP_x\PAR{X_{\exptime }\in \mathrm{d}y;\exptime <\tau_a^- } + \dE_x\SBRA{v(X_{\tau_a^-},\mathrm{d}y );\tau_a^-<\exptime }\\
&=\Um{p}{x}{a}\ind_{y>a}+\dE_x\SBRA{ \mathrm{e}^{-p\tau_a^-}v(X_{\tau_a^-},\mathrm{d}y )}.
\end{align*}
Since $X_{\tau_a^-}<a$ a.s., we deduce from \eqref{eq:va} that for $x\geq a$
\begin{multline*}
v(x,\mathrm{d}y )=\Um{p}{x}{a}\ind_{y>a}+\dE_x\SBRA{ \mathrm{e}^{-p\tau_a^-}\U{p+q}{X_{\tau_a^-}}{a}}\ind_{y\in[0,a]}\\
+{v(a,\mathrm{d}y )\over \pqscale(a)}\dE_x\SBRA{ \mathrm{e}^{-p\tau_a^-}\pqscale(X_{\tau_a^-})} \\
+v(0,\mathrm{d}y )\dE_x\SBRA{ \mathrm{e}^{-p\tau_a^-}\dE_{X_{\tau_a^-}}\SBRA{ \mathrm{e}^{-(p+q)\tau_0^-+\Phi(p)X_{\tau_0^-}};\tau_0^-< \tau_a^+}}\\
+\dE_x\SBRA{ \mathrm{e}^{-p\tau_a^-}\dE_{X_{\tau_a^-}}\SBRA{ \mathrm{e}^{-(p+q)\tau_0^-}\Up{p}{X_{\tau_0^-}}{0};\tau_0^-< \tau_a^+}}\ind_{y<0}.
\end{multline*}
To conclude the proof, we just notice that this last expression is satisfied for any $x\in\dR$.
\end{proof}

%

\subsection{Joint distributions of occupations times and the process sampled at a fixed time}

As in~\cite{loeffenetal2014}, for the sake of compactness of the next result, we introduce the following functions. First, for $p, p+q \geq 0$ and $x \in \mathbb R$, set
\begin{multline}\label{E:wronde}
\wapq(x) =  W^{(p+q)}(x) - q \int_0^a W^{(p+q)}(x-z) \pscale(z) \mathrm{d}z \\
= W^{(p)}(x) + q \int_a^x W^{(p+q)}(x-z)W^{(p)}(z)\mathrm d z ,
\end{multline}
the second equality following from \eqref{eq:sym}. We notice that $ \wapq(x)=\pscale(x)$ for $x\leq a$. 

Secondly, for $p\geq 0$, $q\in\dR$ with $p+q\geq 0$ and $x\in\dR$, set
\begin{equation}\label{eq:H}
\mathcal{H}^\PAR{p,q}(x) = \mathrm{e}^{\Phi(p) x} \left[ 1 + q \int_0^x \mathrm{e}^{-\Phi(p)z} W^{(p+q)}(z) \mathrm{d}z \right] .
\end{equation}
From Equation~\eqref{E:scale_measure_change}, one easily gets that
\begin{equation}\label{eq:HZ}
\mathcal{H}^\PAR{p,q}(x) = \mathrm{e}^{\Phi(p) x} Z^{(q)}_{\Phi(p)} (x) .
\end{equation}
Finally, note that the Laplace transform of $\mathcal{H}^\PAR{p,q}$ on $[0,\infty)$ is explicitly given by
\begin{equation}\label{eq:LTH}
\int_0^{\infty} \mathrm{e}^{-\lambda x}\mathcal{H}^\PAR{p,q}(x) \mathrm{d}x={1\over \lambda-\Phi(p)}\PAR{1+{q\over \psi(\lambda)-p-q}} \, ,\text{ for }\lambda>\Phi(p+q) .
\end{equation}

Here is our main result.
\begin{theorem}\label{main}
Fix $a>0$, $q \geq 0$ and $x\in\dR$. 
For $p > 0$, $y\in\dR$,
\begin{multline*}
\int_0^\infty \mathrm{e}^{-p t} \e_x \left[ \mathrm{e}^{- q \int_0^{t} \ind_{(0,a)} (X_s) \mathrm{d}s } ; X_t \in \mathrm{d}y \right] \mathrm{d}t\\
\hskip -3cm =\mathrm{e}^{-\Phi(p)a}\left( \frac{\mathcal{H}^\PAR{p,q}(x) - q \int_a^x \pscale(x-z) \mathcal{H}^\PAR{p,q}(z) \mathrm{d}z}{\psi'(\Phi(p)) + q \int_0^a \mathrm{e}^{-\Phi(p) z} \mathcal{H}^\PAR{p,q}(z) \mathrm{d}z} \right)\\
\times \left\lbrace \cH^\PAR{p,q}(a-y) - q \int_0^{-y} \cH^\PAR{p,q}(a-y-z) \pscale(z) \mathrm{d}z \right\rbrace \dy \\
- \left\{ \mathcal W_{x-a}^{(p,q)} (x-y) -q \int_0^{-y} \mathcal W_{x-a}^{(p,q)} (x-y-z) \pscale(z) \mathrm{d}z \right\} \dy .
\end{multline*}
\end{theorem}


We extend the result of Theorem~\ref{main} to any finite and then semi-infinite interval, i.e.\ we study the Laplace transform of 
\begin{equation*}
\e_x \left[ \mathrm{e}^{- q \int_0^t \ind_{(a,b)} (X_s) \mathrm{d}s } ; X_t \in \mathrm{d}y \right] ,
\end{equation*}
where $t>0$.

We could re-apply the same methodology (in a quite shorter form) as in the proof of Theorem~\ref{main} for proving the second part of the next result, but we will instead take limits in our main result.
\begin{corollary}\label{main_prop}
Fix $a,b\in\dR$ with $a<b$. For all $x\in\dR$ and $q \geq 0$, we have that, for $p > 0$,
\begin{multline*}
\int_0^\infty \mathrm{e}^{-p t} \e_x \left[ \mathrm{e}^{- q \int_0^{t} \ind_{(a,b)} (X_s) \mathrm{d}s } ; X_t \in \mathrm{d}y \right] \mathrm{d}t \\
\hskip -2cm=
\mathrm{e}^{-\Phi(p)(b-a)} \left( \frac{\mathcal{H}^\PAR{p,q}(x-a) - q \int_{b}^{x} \pscale(x-z) \mathcal{H}^\PAR{p,q}(z-a) \mathrm{d}z}{\psi'(\Phi(p)) + q \int_0^{b-a} \mathrm{e}^{-\Phi(p) z} \mathcal{H}^\PAR{p,q}(z) \mathrm{d}z} \right) \\
\times \left\lbrace \cH^\PAR{p,q}(b-y) - q \int_0^{-y+a} \cH^\PAR{p,q}(b-y-z) \pscale(z) \mathrm{d}z \right\rbrace \dy \\
- \left\{ \mathcal W_{x-b}^{(p,q)} (x-y) -q \int_0^{-y+a} \mathcal W_{x-b}^{(p,q)} (x-y-z) \pscale(z) \mathrm{d}z \right\} \dy,
\end{multline*}
and
\begin{multline}\label{E:semi-infinite}
\int_0^\infty \mathrm{e}^{-p t} \e_x \left[ \mathrm{e}^{- q \int_0^{t} \ind_{(-\infty,b)} (X_s) \mathrm{d}s } ; X_t \in \mathrm{d}y \right] \mathrm{d}t \\
=\BRA{\left( \frac{\Phi(p+q)-\Phi(p)}{q} \right) \cH^\PAR{p+q,-q}(x-b) \cH^\PAR{p,q}(b-y) - \mathcal W_{x-b}^{(p,q)} (x-y)} \dy ,
\end{multline}
where $\lim_{q \to 0} \left( \Phi(p+q)-\Phi(p) \right)/q = 1/\psi' \left( \Phi(p) \right)$ in the case $q=0$.
\end{corollary}

\begin{proof}
From the spatial homogeneity of the process $X$, we have
\[
\e_x \left[ \mathrm{e}^{- q \int_0^T \ind_{(a,b)} (X_s) \mathrm{d}s } ; X_T \in \mathrm{d}y \right]=\e_{x-a} \left[ \mathrm{e}^{- q \int_0^T \ind_{(0,b-a)} (X_s) \mathrm{d}s } ; a+X_T \in \mathrm{d}y \right],
\]
then first result is just an easy consequence of Theorem~\ref{main}.

For the second result, we will take the limit of the first result when $a$ goes to $-\infty$. First note that, from~\eqref{eq:HZ} and \eqref{E:scale_measure_change},
\begin{multline}\label{E:HoverW}
\frac{\mathcal{H}^\PAR{p,q}(x+c)}{\pqscale (c)} = \mathrm{e}^{\Phi(p) x} \frac{Z_{\Phi(p)}^{(q)} (x+c)}{W_{\Phi(p)}^{(q)} (x+c)} \left( \frac{W_{\Phi(p)}^{(q)} (x+c)}{W_{\Phi(p)}^{(q)} (c)} \right) \\
\underset{c \to \infty}{\longrightarrow} \mathrm{e}^{\Phi(p) x} \frac{q}{\Phi_{\Phi(p)} (q)} \mathrm{e}^{\Phi_{\Phi(p)} (q)x} = \mathrm{e}^{\Phi(p+q) x} \frac{q}{\Phi (p+q) - \Phi(p)} ,
\end{multline}
where, to compute the limit, we used the fact that $\lim_{c \to \infty} \zqscale(c)/\qscale (c) = q/\Phi(q)$ and, to get the last expression, we used the results from the beginning of Section 3. As a consequence, using Lebesgue's dominated convergence theorem, we get
\begin{multline}\label{eq:lim1}
\lim_{a\rightarrow -\infty} \frac{\mathcal{H}^\PAR{p,q}(x-a) - q \int_{b}^{x} \pscale(x-z) \mathcal{H}^\PAR{p,q}(z-a) \mathrm{d}z}{\pqscale(-a)}\\
={q\over \Phi(p+q)-\Phi(p)}\PAR{\mathrm{e}^{\Phi(p+q) x} - q \int_{b}^{x} \pscale(x-z) \mathrm{e}^{\Phi(p+q) z}\mathrm{d}z}\\
={q\over \Phi(p+q)-\Phi(p)}\mathrm{e}^{\Phi(p+q) b}\mathcal{H}^\PAR{p+q,-q}(x-b).
\end{multline}


On the other hand, using again Lebesgue's dominated convergence theorem with the limit in~\eqref{E:HoverW} and~\eqref{E:scale_measure_change}, we can write
\begin{multline}\label{eq:lim2}
\lim_{a\rightarrow -\infty} \frac{\mathrm{e}^{\Phi(p)(b-a)}}{\pqscale(-a)}\PAR{\psi'(\Phi(p)) + q \int_0^{b-a} \mathrm{e}^{-\Phi(p) z} \mathcal{H}^\PAR{p,q}(z) \mathrm{d}z}={q^2\over \PAR{\Phi(p+q)-\Phi(p)}^2}\mathrm{e}^{\Phi(p+q)b} ,
\end{multline}
since we also have that
$$
\mathrm{e}^{\Phi(p) (b-a)} \frac{\psi'(\Phi(p))}{\pqscale(-a)} = \mathrm{e}^{\Phi(p) b} \frac{\psi'(\Phi(p))}{W_{\Phi(p)}^{(q)} (-a)} \underset{a \to -\infty}{\longrightarrow}0 .
$$
Finally, combining~\eqref{eq:lim1} and~\eqref{eq:lim2}, we have
\begin{multline*}
\lim_{a\rightarrow -\infty} \mathrm{e}^{-\Phi(p)(b-a)} \frac{\mathcal{H}^\PAR{p,q}(x-a) - q \int_{b}^{x} \pscale(x-z) \mathcal{H}^\PAR{p,q}(z-a) \mathrm{d}z}{\psi'(\Phi(p)) + q \int_0^{b-a} \mathrm{e}^{-\Phi(p) z} \mathcal{H}^\PAR{p,q}(z) \mathrm{d}z} \\
={\PAR{\Phi(p+q)-\Phi(p)}\over q}\mathcal{H}^\PAR{p+q,-q}(x-b).
\end{multline*}
Using the fact that $\pscale(x)=0$ for $x<0$, the proof is complete.
\end{proof}

\begin{remark}
\begin{enumerate}
\item
We notice that $\mathcal W_{0}^{(p,q)} (x)=\pqscale(x)$. Thus, for $x=b$, the expression in Equation \eqref{E:semi-infinite} becomes
\begin{multline*}
\int_0^\infty \mathrm{e}^{-p t} \e_b \left[ \mathrm{e}^{- q \int_0^{t} \ind_{(-\infty,b)} (X_s) \mathrm{d}s } ; X_t \in \mathrm{d}y \right] \mathrm{d}t \\
=\BRA{\left( \frac{\Phi(p+q)-\Phi(p)}{q} \right)  \cH^\PAR{p,q}(b-y) - \pqscale (b-y)} \dy.
\end{multline*}
\item Note that when $q=0$, the result in Equation~\eqref{E:semi-infinite} becomes
$$
\int_0^\infty \mathrm{e}^{-p t} \p_x \left( X_t \in \mathrm{d}y \right) \mathrm{d}t = \left\lbrace \frac{\mathrm{e}^{\Phi(p)(x-y)}}{\psi' \left( \Phi(p) \right)} - W^{(p)} (x-y) \right\rbrace \dy ,
$$
which agrees with  \cite[Corollary 8.9]{kyprianou2006} for the density of the $p$-potential measure of $X$ without killing.
\end{enumerate}
\end{remark}

\section{Pricing step options}

\subsection{General case}
Following the notations introduced in Section \ref{sec:intro-options}, we consider the risk-neutral price process:
$$
S_t = S_0 \mathrm{e}^{X_t} ,
$$
where $X = \{ X_t , t \geq 0\}$ is a spectrally negative L\'evy process.
Recall that the price of a (down-and-out call) step option with knock-out rate $\rho$ and risk-free interest rate $r$ is given by:
\begin{equation}
\label{eq:price}
C(T) :=\mathrm{e}^{-r T} \e \left[ \mathrm{e}^{- \rho \int_0^T \mathbf{1}_{\{S_t \leq L\}} \mathrm{d}t } \left( S_T - K \right)_+ \right] .
\end{equation}
 In what follows, without loss of generality, we take $r=0$. The Laplace transform of the price is then given by
\[
\int_0^\infty \mathrm{e}^{-pT} C(T) \mathrm{d}T 
= \int_{\ln \left( K/S_0 \right)}^\infty \left( S_0 \mathrm{e}^{y} - K \right) \int_0^\infty \mathrm{e}^{-pT} \e \left[ \mathrm{e}^{- \rho \int_0^T \ind_{(-\infty,\ln(L/S_0))} (X_s) \mathrm{d}s } ; X_T \in \mathrm{d}y \right] \mathrm{d}T,
\]
where, by Corollary~\ref{main_prop} or more precisely Equation~\eqref{E:semi-infinite}, we have
\begin{multline*}\label{eq:step-options}
\int_0^\infty \mathrm{e}^{-pT} \e \left[ \mathrm{e}^{- \rho \int_0^T \ind_{(-\infty,\ln(L/S_0))} (X_s) \mathrm{d}s } ; X_T \in \mathrm{d}y \right] \mathrm{d}T \\
=\BRA{\left( \frac{\Phi(p+\rho)-\Phi(p)}{\rho} \right) \cH^\PAR{p+\rho,-\rho}(\ln(S_0/L)) \cH^\PAR{p,\rho}(\ln(L/S_0)-y) - \mathcal W_{\ln(S_0/L)}^{(p,\rho)} (-y)}\dy.\end{multline*}

From a financial point of view, as we are interested in a down-and-out step option, it is natural to look at the case when $S_0>L$ and $K>L$. 
\begin{corollary}\label{coro:price}
In an exponential spectrally negative L\'evy model, the Laplace transform of the price $C(T)$ of a (down-and-out call) step option with $S_0,K>L$ is given by, for $p>0$ such that $\Phi(p)>1$,
\begin{multline*}
\int_0^\infty \mathrm{e}^{-pT} C(T) \mathrm{d}T=\\
\frac{\Phi(p+\rho)-\Phi(p)}{\rho\Phi(p)\PAR{\Phi(p)-1}}\cH^\PAR{p+\rho,-\rho}(\ln(S_0/L))K \PAR{L\over K}^{\Phi(p)}
 - \int^{\ln \left( S_0/K \right)}_0\left(S_0\mathrm{e}^{-y} - K \right)\pscale(y)\dy,
\end{multline*}
with
\[
 \cH^\PAR{p+\rho,-\rho}(\ln(S_0/L))=\PAR{S_0\over L}^{\Phi(p+\rho)}\PAR{1-\rho\int_0^{\ln(S_0/L)}\ee^{-\Phi(p+\rho)y}\pscale(y)\dy}.
\]
\end{corollary}
\begin{proof}
We first notice that, under the assumptions on $S_0$, $K$ and $L$,  on the interval $[\ln(K/S_0),+\infty)$, we have $y>\ln(L/S_0)$ and thus
\[
\cH^\PAR{p,\rho}(\ln(L/S_0)-y)=\PAR{L\over S_0}^{\Phi(p)}\ee^{- \Phi(p) y}\quad \text{and}\quad \mathcal W_{\ln(S_0/L)}^{(p,\rho)} (-y)=\pscale\PAR{-y}.
\]

Then, when $K>L$, we have, for $p>0$ such that $\Phi(p)>1$,
\begin{multline*}
\int_0^\infty \mathrm{e}^{-pT} C(T) \mathrm{d}T\\
=\left( \frac{\Phi(p+\rho)-\Phi(p)}{\rho} \right) \cH^\PAR{p+\rho,-\rho}(\ln(S_0/L)) \PAR{L\over S_0}^{\Phi(p)} \int_{\ln \left( K/S_0 \right)}^\infty\left( S_0 \mathrm{e}^{y} - K \right)\ee^{-\Phi(p)y}\dy\\
-  \int_{\ln \left( K/S_0 \right)}^{\infty}\left( S_0 \mathrm{e}^{y} - K \right)\pscale(-y)\dy\\
=\frac{\Phi(p+\rho)-\Phi(p)}{\rho\Phi(p)\PAR{\Phi(p)-1}}\cH^\PAR{p+\rho,-\rho}(\ln(S_0/L))K \PAR{L\over K}^{\Phi(p)}
 - \int_{\ln \left( K/S_0 \right)}^0\left(S_0\mathrm{e}^{y} - K \right)\pscale(-y)\dy,
\end{multline*}
which concludes the proof.
\end{proof}

Of course, using Corollary \ref{main_prop}, we could easily derive similar expressions for down-and-out put step options, as well as for up-and-out call/put step options.

\subsection{Particular case of a Lévy jump-diffusion process with hyper-exponential jumps}

We now present the case of a L\'evy jump-diffusion process where the jump distribution is a mixture of exponential distributions, as in \cite{loeffenetal2014}. In other words, let
\begin{equation}\label{E:HE-JD}
X_t = \realsift t + \sigma B_t - \sum_{i=1}^{N_t} \xi_i ,
\end{equation}
where $\sigma \geq 0$, $\realsift \in \mathbb R$, $B=\BRA{B_t,t \geq 0}$ is a Brownian motion, $N=\BRA{N_t,t \geq 0}$ is a Poisson process with intensity $\eta>0$, and $\{\xi_1,\xi_2,\ldots\}$ are iid (positive) random variables with common probability density function given by
\begin{equation*}
f_\xi (y) = \left( \sum_{i=1}^n a_i \alpha_i \mathrm{e}^{-\alpha_i y} \right) \ind_{\{y>0\}} ,
\end{equation*}
where $n$ is a positive integer, $0<\alpha_1<\alpha_2<\ldots<\alpha_n$ and $\sum_{i=1}^n a_i=1$, where $a_i>0$ for all $i=1,\ldots,n$. All of the aforementioned objects are mutually independent. 

The Laplace exponent of $X$ is then clearly given by
$$
\psi(\lambda) = \realsift \lambda  + \frac12\sigma^2 \lambda ^2 + \eta \left( \sum_{i=1}^n \frac{a_i \alpha_i}{\lambda +\alpha_i} - 1 \right) , 
$$
for $\lambda >-\alpha_1$. In this case,
$$
\e \left[ X_1 \right] = \psi'(0+) = \realsift - \eta \sum_{i=1}^n \frac{a_i}{\alpha_i} .
$$

We notice that $\lambda \mapsto \left( \psi(\lambda)-q \right)^{-1}$ is a rational fraction of the form $P(\lambda)/Q(\lambda)$, where $P$ and $Q$ are polynomial functions given by
$$
P(\lambda) = \prod_{i=1}^n (\lambda+\alpha_i) \quad \text{and} \quad Q(\lambda) = \left( \psi(\lambda)-q \right) P(\lambda) .
$$
Note that if $\sigma > 0$, then $Q$ is a polynomial of degree $n+2$ with $\sigma^2/2$ as the coefficient of the leading term, and if $\sigma = 0$, then $Q$ is a polynomial of degree $n+1$ with $\realsift$ as the coefficient of the leading term. In all cases, $P$ is a polynomial of degree $n$. To ease the presentation, set $N=(n+1)+\ind_{\{\sigma>0\}}$.
For $q \geq 0$, the function $\lambda \mapsto \left( \psi(\lambda)-q \right)^{-1}$ has $N$ simple poles such that $\theta_{N}^{(q)}  < \theta_{N-1}^{(q)} < \ldots < \theta_2^{(q)} \leq 0 \leq \theta_1^{(q)}$ with $\theta_1^{(q)}=\Phi(q)$. Studying the functions $\lambda \mapsto \psi(\lambda)-q$ and $\lambda \mapsto Q(\lambda)$ extended on $\mathbb{R} \setminus \{-\alpha_n, \dots, -\alpha_1\}$, we deduce that, for $\sigma > 0$, 
$$
\theta_{n+2}^{(q)}  < -\alpha_n <  \theta_{n+1}^{(q)} < -\alpha_{n-1} < \theta_{n}^{(q)} \ldots <  -\alpha_1 < \theta_2^{(q)} \leq 0 \leq \theta_1^{(q)} ,
$$
and, for $\sigma = 0$,
$$
-\alpha_n <  \theta_{n+1}^{(q)} < -\alpha_{n-1} < \theta_{n}^{(q)} \ldots <  -\alpha_1 < \theta_2^{(q)} \leq 0 \leq \theta_1^{(q)} ,
$$
where, in both cases, we have one of the following three cases:
\begin{itemize}
\item $\theta_2^{(q)} < \theta_1^{(q)} = \Phi(q) = 0$ if and only if $q=0$ and $\psi'(0) > 0$,
\item $\theta_2^{(q)} = \theta_1^{(q)} = \Phi(q) = 0$ if and only if $q=0$ and $\psi'(0) = 0$,
\item $\theta_2^{(q)} = 0 < \theta_1^{(q)} = \Phi(q)$ if and only if $q=0$ and $\psi'(0) < 0$.
\end{itemize}
For more details, see e.g.\ \cite{cai_2009} where a similar analysis is undertaken for a closely related jump-diffusion process. For simplicity, in what follows, we assume that either $q>0$ or $q=0$ and $\psi'(0) \neq 0$, in which cases $\theta_2^{(q)} \neq \theta_1^{(q)}$. Consequently, using the classical decomposition of rational fractions and noticing that 
\[
{P\PAR{\theta_i^{(q)}}\over Q'\PAR{\theta_i^{(q)}}}={1\over \psi'\PAR{\theta_i^{(q)}}},
\]
 we can write
\begin{equation}\label{partialfraction}
\frac{1}{\psi(\lambda)-q} = \sum_{i=1}^{N} {1\over \psi'\PAR{\theta_i^{(q)}}}{1 \over  \left(\lambda -\theta_i^{(q)} \right)}.
\end{equation}
In conclusion, by Laplace inversion, we have, for $x \geq 0$,
\begin{align}\label{eq:Wexple}
W^{(q)}(x) &= \sum_{i=1}^{N}{ \mathrm{e}^{ \theta_i^{(q)} x}\over \psi'\PAR{\theta_i^{(q)}}} , \\
Z^{(q)}(x) &=
\begin{cases}
q \sum_{i=1}^{N} { \mathrm{e}^{ \theta_i^{(q)} x}\over \theta_i^{(q)} \psi'\PAR{\theta_i^{(q)}}} & \text{if $q>0$,} \\
1 & \text{if $q=0$,}
\end{cases}
\end{align}
since taking $\lambda=0$ in~\eqref{partialfraction} gives $q \sum_{i=1}^{N} \PAR{ \theta_i^{(q)} \psi'\PAR{\theta_i^{(q)}}}^{-1}=1$.

\begin{remark} We recover well-known expressions of scale functions in two specific cases (see e.g.\ \cite{kuznetsovetal2012}).
When $\sigma>0$ and $n=0$, $X$ is a Brownian motion with drift and, for $x\geq 0$,
\[
W^{(q)}(x) ={\mathrm{e}^{\theta_1^{(q)}x}\over \sqrt{\Delta_q}}-{\mathrm{e}^{\theta_2^{(q)} x}\over\sqrt{\Delta_q}},
\]
with $\theta_1^{(q)}={1\over \sigma^2}\PAR{\sqrt{\Delta_q}-\realsift}$, $\theta_2^{(q)}={-1\over \sigma^2}\PAR{\sqrt{\Delta_q}+\realsift}$ and $\Delta_q=\realsift^2+2\sigma^2 q$.

When $\sigma=0$ and $n=1$, $X$ is a compound Poisson process with drift and exponential jumps and,  for $x\geq 0$,
\[
W^{(q)}(x) ={\alpha+\theta_1^{(q)}\over \sqrt{\Delta_q}}\mathrm{e}^{\theta_1^{(q)}x}-{\alpha+\theta_2^{(q)}\over \sqrt{\Delta_q}}\mathrm{e}^{\theta_2^{(q)} x},
\]
with $\theta_1^{(q)}={1\over2\realsift}\PAR{q+\eta-\realsift\alpha+\sqrt{\Delta_q}}$, $\theta_2^{(q)}={1\over2\realsift}\PAR{q+\eta-\realsift\alpha -\sqrt{\Delta_q}}$ and $\Delta_q=(q+\eta-\realsift\alpha)^2+4\realsift\alpha q$.
\end{remark}

Now, let us explicitly compute the Laplace transform of the price $C(T)$, given by Corollary \ref{coro:price}, for this particular SNLP.

We first give an expression of $\mathcal{H}^\PAR{p,q}$. Let us recall, from \eqref{eq:HZ}, that $\mathcal{H}^\PAR{p,q}(x) = \mathrm{e}^{\Phi(p) x} Z^{(q)}_{\Phi(p)} (x) $ with $\Phi(p)=\theta_1^{(p)}$. Denoting by $\theta_{\Phi(p),i}^{(q)} $ the poles of $\PAR{\psi_{\Phi(p)} (\lambda)-q}^{-1}$ and using \eqref{eq:psic},
 we have  $\theta_{\Phi(p),i}^{(q)}=\theta_i^{(p+q)}-\Phi(p)$ and $ \psi' _{\Phi(p)}\left( \theta_{\Phi(p),i}^{(q)} \right)=\psi'\left( \theta_i^{(p+q)} \right)$.
 Consequently, for $q>0$ and $x\geq 0$, 
\begin{equation}\label{eq:Hexple}
\mathcal{H}^\PAR{p,q}(x) =  q \sum_{i=1}^{N} \frac{ \mathrm{e}^{ \theta_i^{(p+q)} x} }{\PAR{\theta_i^{(p+q)}-\theta_1^{(p)}}\psi'\left( \theta_i^{(p+q)} \right) },
\end{equation}
and for $x<0$, $\mathcal{H}^\PAR{p,q}(x) =\exp\PAR{ \theta_1^{(p)} x}$.

We also compute, for $a<x$,
\[
\mathcal{W}_a^{(p,q)}(x)= q \sum_{i,j=1}^{N} \frac{ \exp\PAR{ \theta_i^{(p+q)} x+ \PAR{\theta_j^{(p)}-\theta_i^{(p+q)}} a} }{\PAR{\theta_i^{(p+q)}-\theta_j^{(p)}}\psi'\left( \theta_i^{(p+q)} \right)\psi'\left( \theta_j^{(p)} \right) }.
\]
Note that $\mathcal{W}_a^{(p,q)}$ had already been computed in \cite{loeffenetal2014}. However, our expression here is slightly simpler because we made one more simplification.

\begin{corollary}
In a Lévy jump-diffusion model with hyper-exponential jumps, the Laplace transform of the price $C(T)$ of a (down-and-out call) step option with $S_0,K>L$ is given by, for $p>0$ such that $\Phi(p)>1$, 
\begin{multline*}
\int_0^\infty \mathrm{e}^{-pT} C(T) \mathrm{d}T=\\
K \sum_{i=2}^{N}\frac{ 1 }{\psi'\left( \theta_i^{(p)} \right) }\left[  \frac{ \Phi{(p+\rho)}-\Phi{(p)} }{\PAR{\Phi{(p+\rho)}-\theta_i^{(p)}}\Phi{(p)}\PAR{\Phi{(p)}-1}}\PAR{L\over K}^{\Phi{(p)}}\PAR{S_0\over L}^{ \theta_i^{(p)}}\right.\\
 \left.
-{1\over \theta_i^{(p)}\PAR{\theta_i^{(p)}-1}}\PAR{S_0\over K}^{\theta_i^{(p)}}\right],
 \end{multline*}
when $S_0>K>L$, and
\[
\int_0^\infty \mathrm{e}^{-pT} C(T) \mathrm{d}T
 =K\frac{\Phi{(p+\rho)}-\Phi{(p)}}{\Phi{(p)}\PAR{\Phi{(p)}-1}}\sum_{i=1}^{N} \frac{ 1 }{\PAR{\Phi{(p+\rho)}-\theta_i^{(p)}}\psi'\left( \theta_i^{(p)} \right) }\PAR{L\over K}^{\Phi{(p)}} \PAR{S_0\over L}^{ \theta_i^{(p)}},
 \]
when $K\geq S_0>L$.
\end{corollary}
\begin{proof}
From \eqref{eq:Hexple}, we deduce that, for $S_0>L$,
\[
\cH^\PAR{p+\rho,-\rho}(\ln(S_0/L))=\rho \sum_{i=1}^{N} \frac{ 1 }{\PAR{\theta_1^{(p+\rho)}-\theta_i^{(p)}}\psi'\left( \theta_i^{(p)} \right) }\PAR{S_0\over L}^{ \theta_i^{(p)}}.
\]
Using Expression \eqref{eq:Wexple}, we also deduce that, when $S_0>K$,
\begin{equation*}
\int^{\ln \left( S_0/K \right)}_0\left(S_0\mathrm{e}^{-y} - K \right)\pscale(y)\dy=K\sum_{i=1}^{N}{1\over \theta_i^{(p)}\PAR{\theta_i^{(p)}-1}\psi'\PAR{\theta_i^{(p)}}}\PAR{S_0\over K}^{\theta_i^{(p)}},
\end{equation*}
otherwise, when $K\geq S_0$, the integral is equal to zero.
Then, from Corollary \ref{coro:price}, for $S_0>K$,
\begin{multline*}
\int_0^\infty \mathrm{e}^{-pT} C(T) \mathrm{d}T\\
=\frac{\theta_1^{(p+\rho)}-\theta_1^{(p)}}{\rho\theta_1^{(p)}\PAR{\theta_1^{(p)}-1}}\cH^\PAR{p+\rho,-\rho}(\ln(S_0/L))K \PAR{L\over K}^{\theta_1^{(p)}}
 - \int^{\ln \left( S_0/K \right)}_0\left(S_0\mathrm{e}^{-y} - K \right)\pscale(y)\dy\\
 =K\frac{\theta_1^{(p+\rho)}-\theta_1^{(p)}}{\theta_1^{(p)}\PAR{\theta_1^{(p)}-1}} \sum_{i=1}^{N} \frac{ 1 }{\PAR{\theta_1^{(p+\rho)}-\theta_i^{(p)}}\psi'\left( \theta_i^{(p)} \right) }\PAR{S_0\over L}^{ \theta_i^{(p)}}\PAR{L\over K}^{\theta_1^{(p)}}
 \\- K\sum_{i=1}^{N}{1\over \theta_i^{(p)}\PAR{\theta_i^{(p)}-1}\psi'\PAR{\theta_i^{(p)}}}\PAR{S_0\over K}^{\theta_i^{(p)}}\\
 =\sum_{i=2}^{N}\frac{ 1 }{\psi'\left( \theta_i^{(p)} \right) }S_0^{\theta_i^{(p)}}\SBRA{  \frac{ \theta_1^{(p+\rho)}-\theta_1^{(p)} }{\PAR{\theta_1^{(p+\rho)}-\theta_i^{(p)}}\theta_1^{(p)}\PAR{\theta_1^{(p)}-1}}L^\PAR{ \theta_1^{(p)}- \theta_i^{(p)}}K^{1-\theta_1^{(p)}}
-{1\over \theta_i^{(p)}\PAR{\theta_i^{(p)}-1}}K^{1-\theta_i^{(p)}}}.
 \end{multline*}
 The result follows since $\theta_1^{(s)}=\Phi(s)$ for all $s\geq0$.
\end{proof}

\section{Acknowledgements}

We would like to thank Xiaowen Zhou for his help and comments on an earlier version of this paper. We would also like to thank the CNRS, its UMI 3457 and the Centre de recherches math\'ematiques (CRM) for providing the research infrastructure.

Funding in support of this work was provided by the Natural Sciences and Engineering Research Council of Canada (NSERC), the Fonds de recherche du Qu\'ebec - Nature et technologies (FRQNT) and the Insti\-tut de finance math\'ematique de Montr\'eal (IFM2).

\begin{appendices}
\section{Technical results}\label{section:tech}

It is easy to show that the Laplace transform of $x \mapsto \int_0^x \pscale(x-y) \mathcal{H}^\PAR{p,q}(y) \mathrm{d}y$ is given by
$$
\lambda \mapsto \left( \frac{1}{\lambda - \Phi(p)} \right) \left( \frac{1}{\psi(\lambda) - (p+q)} \right) .
$$
Therefore,
$$
\int_0^x \pscale(x-z) \mathcal{H}^\PAR{p,q}(z) \mathrm{d}z = \int_0^x \mathrm{e}^{\Phi(p) (x-z)} W^{(p+q)}(z) \mathrm{d}z ,
$$
from which we deduce, in the spirit of Equation~\eqref{E:wronde} and Equation~(6) in \cite{loeffenetal2014},
\begin{equation}\label{E:hronde}
\mathcal{H}^\PAR{p,q}(x) - q \int_a^x \pscale(x-y) \mathcal{H}^\PAR{p,q}(y) \mathrm{d}y = \mathrm{e}^{\Phi(p)x} + q \int_0^a \pscale(x-y) \mathcal{H}^\PAR{p,q}(y) \mathrm{d}y .
\end{equation}
Note that, we can use the last relation to rewrite our main result in Theorem \ref{main} and Lemma \ref{lem:techlemma3}.\\


The next result is well known (see e.g. \cite{kyprianou2006,landriaultetal2014}), but we rewrite it in terms of $\mathcal{H}^\PAR{p,q}$ for future usage.
\begin{lemma}\label{lem:techlemma1}
For $x \in \reals$, $a\leq b$, $p,q\geq0$, we have 
\begin{equation*}
\e_x \left[ \mathrm{e}^{-(p+q) \tau_a^- + \Phi(p) X_{\tau_a^-} }; \tau_a^- < \tau_b^+ \right] = \mathrm{e}^{\Phi(p)a}\mathcal{H}^\PAR{p,q}(x-a) - \frac{\mathcal{H}^\PAR{p,q}(b-a)}{W^{(p+q)}(b-a)} \mathrm{e}^{\Phi(p)a} W^{(p+q)}(x-a) .
\end{equation*}
Moreover, for  $x,a \in \reals$, $p\geq0$ and $q>0$, we have 
\begin{equation*}
\e_x \left[ \mathrm{e}^{-(p+q) \tau_a^- + \Phi(p) X_{\tau_a^-} }; \tau_a^- < \infty \right] =  \mathrm{e}^{\Phi(p) a} \mathcal{H}^\PAR{p,q}(x-a) - {q \over \Phi(p+q)-\Phi(p)} \mathrm{e}^{\Phi(p) a} \pqscale(x-a)
\end{equation*}
and, when $q \to 0$, we get
\begin{equation*}
\e_x \left[ \mathrm{e}^{-p \tau_a^- + \Phi(p) X_{\tau_a^-} }; \tau_a^- < \infty \right] =  \mathrm{e}^{\Phi(p)x} - \psi'(\Phi(p)) \mathrm{e}^{\Phi(p)a} \pscale(x-a).
\end{equation*}
\end{lemma}

The next lemma is an immediate consequence of Lemma 2.2 in \cite{loeffenetal2014}.
\begin{lemma}\label{lem:techlemma2}
Let $p,q \geq 0$. For $a \leq b$ and $x,y \in \dR$, we have
\begin{multline*}
\e_x \left[ \mathrm{e}^{-p\tau_a^-} W^{(q)}(X_{\tau_a^-} -y) ; \tau_a^-<\tau_b^+ \right] = W^{(p)}(x-y)+(q-p)\int_{0}^{a-y} \pscale(x-y-z) W^{(q)}(z)\mathrm{d}z \\
-{\pscale(x-a) \over \pscale(b-a)} \PAR{W^{(p)}(b-y) + (q-p) \int_{0}^{a-y} \pscale(b-y-z) W^{(q)}(z) \mathrm{d}z} ,
\end{multline*}
and
\begin{multline*}
\e_x \left[ \mathrm{e}^{-p\tau_a^-} \qscale(X_{\tau_a^-}-y) ; \tau_a^-<+\infty \right] \\
= \pscale(x-y) + (q-p) \int_{0}^{a-y} \pscale(x-y-z) \qscale(z) \mathrm{d}z \\
- \pscale(x-a) \cH^\PAR{p,q-p}(a-y).
\end{multline*}
\end{lemma}

\begin{proof}
First, let's notice that, by spatial homogeneity of $X$, we have
\begin{align}\label{eq:trans1}
\e_x \left[ \mathrm{e}^{-p\tau_a^-} \qscale(X_{\tau_a^-}-y) ; \tau_a^-<\tau_b^+ \right]&=\e_{x-y} \left[ \mathrm{e}^{-p\tau_{a-y}^-} \qscale(X_{\tau_{a-y}^-}) ; \tau_{a-y}^-<\tau_{b-y}^+ \right],\\
\label{eq:trans2}
\e_x \left[ \mathrm{e}^{-p\tau_a^-} \qscale(X_{\tau_a^-}-y) ; \tau_a^-<+\infty \right]&=\e_{x-y}\left[ \mathrm{e}^{-p\tau_{a-y}^-} \qscale(X_{\tau_{a-y}^-}) ; \tau_{a-y}^-<+\infty \right].
\end{align}
So, it suffices to prove the lemma for the case $y=0$.

\medskip
Let $y=0$. We know from Lemma 2.2 in \cite{loeffenetal2014} that
\begin{multline*}
\e_x \left[ \mathrm{e}^{-p\tau_a^-} \qscale(X_{\tau_a^-}) ; \tau_a^-<\tau_b^+ \right] = \qscale(x)-(q-p)\int_{a}^{x} \pscale(x-z)\qscale(z)\mathrm{d}z\\-{\pscale(x-a)\over \pscale(b-a)}\PAR{\qscale(b)-(q-p)\int_{a}^{b}\pscale(b-z)\qscale(z)\mathrm{d}z}. 
\end{multline*}
Using Relation \eqref{eq:sym},
the expression becomes:
\begin{multline*}
\e_x \left[ \mathrm{e}^{-p\tau_a^-} \qscale(X_{\tau_a^-}) ; \tau_a^-<\tau_b^+ \right] = \pscale(x)+(q-p)\int_{0}^{a} \pscale(x-z)\qscale(z)\mathrm{d}z\\-{\pscale(x-a)\over \pscale(b-a)}\PAR{\pscale(b)+(q-p)\int_{0}^{a}\pscale(b-z)\qscale(z)\mathrm{d}z}. 
\end{multline*}

Since $\e_x \left[ \mathrm{e}^{-p\tau_a^-} \qscale(X_{\tau_a^-} ) ; \tau_a^-<\infty \right] =\lim_{b\rightarrow+\infty}\e_x \left[ \mathrm{e}^{-p\tau_a^-} \qscale(X_{\tau_a^-}) ; \tau_a^-<\tau_b^+ \right]$, ${\pscale(b)\over \pscale(b-a)} = \mathrm{e}^{\Phi(p) a}{W_{\Phi(p)}(b)\over W_{\Phi(p)}(b-a)}$ and ${\pscale(b-z)\over \pscale(b-a)} = \mathrm{e}^{\Phi(p)(a-z)}{W_{\Phi(p)}(b-z)\over W_{\Phi(p)}(b-a)}$, by  Lebesgue's dominated convergence theorem we have
\begin{multline*}
\e_x \left[ \mathrm{e}^{-p\tau_a^-} \qscale(X_{\tau_a^-} ) ; \tau_a^-<+\infty \right] = \pscale(x)+(q-p)\int_{0}^{a} \pscale(x-z)\qscale(z)\mathrm{d}z\\- \mathrm{e}^{\Phi(p)(a)}\pscale(x-a)\PAR{1+(q-p)\int_{0}^{a} \mathrm{e}^{-\Phi(p)z}\qscale(z)\mathrm{d}z}. 
\end{multline*}
\end{proof}
Using the same tools as in the previous lemma, we also have the following result for $\zqscale$.
\begin{lemma}\label{lem:techlemma2b}
Let $p,q \geq 0$. For $a \leq b$ and $x,y \in \dR$, we have
\begin{multline*}
\e_x \left[ \mathrm{e}^{-p\tau_a^-} \zqscale(X_{\tau_a^-} -y) ; \tau_a^-<\tau_b^+ \right] = \zqscale(x-y)+(q-p)\int_{0}^{a-y} \pscale(x-y-z)\zqscale(z)\mathrm{d}z \\
-{\pscale(x-a)\over \pscale(b-a)}\PAR{\zqscale(b-y)+(q-p)\int_{0}^{a-y}\pscale(b-y-z) \zqscale(z)\mathrm{d}z}. 
\end{multline*}
and
%
\begin{multline*}
\e_x \left[ \mathrm{e}^{-p\tau_a^-} \zqscale(X_{\tau_a^-}-y) ; \tau_a^-<+\infty \right] \\
= \zpscale(x-y) + (q-p) \int_{0}^{a-y} \pscale(x-y-z) \zqscale(z) \mathrm{d}z \\
- \pscale(x-a) \mathrm{e}^{\Phi(p) (a-y)} \left[ \frac{p}{\Phi(p)} + (q-p) \int_0^{a-y} \mathrm{e}^{\Phi(p) z} \zqscale(z) \mathrm{d}z \right] ,
\end{multline*}
where $\lim_{p \to 0} \frac{p}{\Phi(p)} \to \psi'(0) \vee 0$ in the case $p=0$.
\end{lemma}

%

\begin{lemma}\label{lem:techlemma3}
For all $a,x\in\reals$, $p\geq0$ and $p+q\geq0$,
\begin{multline*}
\e_x \left[ \mathrm{e}^{-p \tau_a^-} \mathcal{H}^\PAR{p,q} \left( X_{\tau_a^-} \right) ; \tau_a^- < \infty \right] \\
= \mathcal{H}^\PAR{p,q}(x) - q \int_a^x \pscale(x-z) \mathcal{H}^\PAR{p,q}(z) \mathrm{d}z \\
- \pscale(x-a) \mathrm{e}^{\Phi(p)a} \left( \psi'(\Phi(p)) + q \int_0^a \mathrm{e}^{-\Phi(p) z} \mathcal{H}^\PAR{p,q}(z) \mathrm{d}z \right).
\end{multline*}
\end{lemma}

\begin{proof}

First, using \eqref{eq:HZ}, we note that
\begin{align*}
\e_x \left[ \mathrm{e}^{-p \tau_a^-} \mathcal{H}^\PAR{p,q} \left( X_{\tau_a^-} \right) ; \tau_a^- < \infty \right] &= \e_x \left[ \mathrm{e}^{-p \tau_a^- + \Phi(p) X_{\tau_a^-}} Z_{\Phi(p)}^{(q)} \left( X_{\tau_a^-} \right) ; \tau_a^- < \infty \right] \\
&= \mathrm{e}^{\Phi(p) x} \e_x^{\Phi(p)} \left[ Z_{\Phi(p)}^{(q)} \left( X_{\tau_a^-} \right) ; \tau_a^- < \infty \right] \\
&= \mathrm{e}^{\Phi(p) x} \e_x \left[ Z_{\Phi(p)}^{(q)} \left( Y_{ \nu_a^-} \right) ; \nu_a^- < \infty \right] ,
\end{align*}
where $Y$ is the SNLP obtained from $X$ by the change of measure with coefficient $\Phi(p)$ and $\nu_a^- = \inf \{t > 0 \colon Y_t < a \}$. Therefore, we can apply Lemma~\ref{lem:techlemma2b} and write
\begin{multline*}
\e_x \left[ \mathrm{e}^{-p \tau_a^-} \mathcal{H}^\PAR{p,q} \left( X_{\tau_a^-} \right) ; \tau_a^- < \infty \right] \\
= \zscale_{\Phi(p)}(x) + q \int_{0}^{a} \scale_{\Phi(p)}(x-y) \zqscale_{\Phi(p)}(y) \mathrm{d}y \\
- \scale_{\Phi(p)}(x-a) \mathrm{e}^{\Phi_{\Phi(p)}(0) a} \left[ \psi_{\Phi(p)}'(0) + q \int_0^{a} \mathrm{e}^{\Phi_{\Phi(p)}(0) y} \zqscale_{\Phi(p)}(y) \mathrm{d}y \right] .
\end{multline*}
since $\psi_{\Phi(p)}'(0) = \psi'(\Phi(p)) \geq 0$ and $\Phi_{\Phi(p)}(0)=0$. The result follows from the discussion at the end of Section \ref{sec:scaleF}.

\end{proof}

\section{Proofs of the main results}

\subsection{Proof of Theorem~\ref{main} when $X$ is of BV}

First, we assume that $X$ has paths of bounded variation (BV).\\
%


In order to simplify the manipulations, we introduce the following (temporary) quantities:
\begin{align*}
A(x) &:= \dE_x\SBRA{ \mathrm{e}^{-p\tau_a^-} \dE_{X_{\tau_a^-}}\SBRA{ \mathrm{e}^{-(p+q)\tau_0^-+\Phi(p)X_{\tau_0^-}};\tau_0^-< \tau_a^+}},\\
B(x;a) &:= \dE_x\SBRA{ \mathrm{e}^{-p\tau_a^-}\pqscale(X_{\tau_a^-})},\\
C(x,\dy) &:= \dE_x\SBRA{ \mathrm{e}^{-p\tau_a^-}\U{p+q}{X_{\tau_a^-}}{a}},\\
D(x,\dy) &:= \dE_x\SBRA{ \mathrm{e}^{-p\tau_a^-}\dE_{X_{\tau_a^-}}\SBRA{ \mathrm{e}^{-(p+q)\tau_0^-}\Up{p}{X_{\tau_0^-}}{0};\tau_0^-< \tau_a^+}} ,
\end{align*}
where, as mentioned previously, all these expectations are taken over the set $\{\tau_a^- < \infty \}$. At this point, let's note that these quantities can be made more explicit (in terms of scale functions) using results from Section~\ref{section:tech}; this will be done later on.\\

Then, using Lemma~\ref{lem:vgeneral}, we can write, for all $x\in\dR$,
\begin{multline}\label{eq:vsimple0}
v(x,\mathrm{d}y )=A(x)v(0,\mathrm{d}y )+{B(x;a)\over \pqscale(a)}v(a,\mathrm{d}y ) \\
+\Um{p}{x}{a}\ind_{y>a}+C(x,\dy)\ind_{y\in[0,a]} + D(x,\dy)\ind_{y<0} .
\end{multline}
Therefore, the quantities $v(0,\mathrm{d}y )$ and $v(a,\mathrm{d}y )$ satisfy the following $2 \times 2$ linear system: 
\begin{equation*}
\begin{cases}
\PAR{1-A(0)}v(0,\mathrm{d}y )-{\pqscale(0)\over\pqscale(a)}v(a,\mathrm{d}y ) =C(0,\dy)\ind_{y\in[0,a]}
+D(0,\dy)\ind_{y<0} , \\[0.05in]
-A(a)v(0,\mathrm{d}y )+\PAR{1-{B(a;a)\over \pqscale(a)}}v(a,\mathrm{d}y ) = \Um{p}{a}{a}\ind_{y>a}+C(a,\dy)\ind_{y\in[0,a]}+D(a,\dy)\ind_{y<0}.
\end{cases}
\end{equation*}
Our aim is to exhibit the values of $v(0,\dy)$ and $v(a,\dy)$ from this linear system using results of Section \ref{section:tech} in order to get an explicit expression of $v(x,\dy)$.

From 
Lemma \ref{lem:techlemma1}, we notice that
\begin{equation}\label{eq:A}
A(x)=\dE_x\SBRA{ \mathrm{e}^{-p\tau_a^-}\cH^\PAR{p,q}(X_{\tau_a^-})}-{\cH^\PAR{p,q}(a)\over \pqscale(a)}B(x;a),
\end{equation}
and from
Lemma~\ref{lem:techlemma2}, we have
\begin{equation}\label{eq:B}
B(x;a)=
\pscale(x)+q\int_{0}^{a} \pscale(x-z)\pqscale(z)\mathrm{d}z-\pscale(x-a)\cH^\PAR{p,q}(a).
\end{equation}
We deduce that
\begin{equation}\label{eq:A0}
A(0)=1-{\pqscale(0)\over\pqscale(a)}\cH^\PAR{p,q}(a),
\end{equation}
and, using Equation~\eqref{eq:sym}, that
\begin{equation}\label{eq:Ba}
B(a;a)=\pqscale(a)-\pscale(0)\cH^\PAR{p,q}(a) .
\end{equation}

Then, since $\pqscale(0) > 0$, the linear system can be written as
\begin{equation*}
\begin{cases}
\cH^\PAR{p,q}(a)v(0,\mathrm{d}y )-v(a,\mathrm{d}y )={\pqscale(a)\over\pqscale(0)}\SBRA{C(0,\dy)\ind_{y\in[0,a]}
+D(0,\dy)\ind_{y<0}} ,\\[0.05in]
-A(a)v(0,\mathrm{d}y )+{\pscale(0)\over \pqscale(a)}\cH^\PAR{p,q}(a)v(a,\mathrm{d}y )=\Um{p}{a}{a}\ind_{y>a}+C(a,\dy)\ind_{y\in[0,a]}+D(a,\dy)\ind_{y<0}.
\end{cases}
\end{equation*}

The determinant of the matrix related to this linear system is equal to
\begin{equation}\label{eq:deltaA}
\Delta={ \pscale(0) \over \pqscale(a)}\PAR{\cH^\PAR{p,q}(a)}^2-A(a),
\end{equation}
with, from Expressions \eqref{eq:A}, \eqref{eq:Ba} and from Lemma \ref{lem:techlemma3}, 
\begin{align}\label{eq:Aa}
A(a)
&=- \pscale(0)\mathrm{e} ^{\Phi(p) a}\left( \psi'(\Phi(p)) + q \int_0^a \mathrm{e}^{-\Phi(p) z} \mathcal{H}^\PAR{p,q}(z) \mathrm{d}z \right)+{\pscale(0)\over \pqscale(a)}\PAR{\cH^\PAR{p,q}(a)}^2.
\end{align}

We deduce that
\begin{equation}\label{eq:delta}
\Delta= \pscale(0)\mathrm{e} ^{\Phi(p) a} \left( \psi'(\Phi(p)) + q \int_0^a \mathrm{e}^{-\Phi(p) z} \mathcal{H}^\PAR{p,q}(z) \mathrm{d}z \right).
\end{equation}
Since $\psi$ is increasing on $[\Phi(0),+\infty)$ and $\pscale(0)=1/c>0$ when $X$ has paths of bounded variation, the determinant $\Delta$ is not equal to zero and there is a unique solution to the linear system satisfied by $v(0,\mathrm{d}y )$ and $v(a,\mathrm{d}y )$. To solve the system we discuss according to the value of $y$.

We first notice, using Equation \eqref{eq:A}  and the first equation of the linear system, that Equation \eqref{eq:vsimple0} can be written in the following way
\begin{multline}\label{eq:vsimple}
v(x,\mathrm{d}y )=\dE_x\SBRA{ \mathrm{e}^{-p\tau_a^-}\cH^\PAR{p,q}(X_{\tau_a^-})} v(0,\dy)
+\Um{p}{x}{a}\ind_{y>a}\\
+\SBRA{C(x,\dy)-{B(x;a)\over \pqscale(0)}C(0,\dy)}\ind_{y\in[0,a]} + \SBRA{D(x,\dy)-{B(x;a)\over \pqscale(0)}D(0,\dy)}\ind_{y<0},
\end{multline}
with, by Lemma \ref{lem:techlemma3},
\begin{equation}\label{eq:EH}
\dE_x\SBRA{ \mathrm{e}^{-p\tau_a^-}\cH^\PAR{p,q}(X_{\tau_a^-})}= \mathcal{H}^\PAR{p,q}(x) - q \int_a^x \pscale(x-z) \mathcal{H}^\PAR{p,q}(z) \mathrm{d}z  - {\pscale(x-a)\over \pscale(0)} \Delta.
\end{equation}
\medskip

{\bf Case 1}: We first consider  $y>a$. The linear system satisfied by $v(0,\mathrm{d}y )$ and $v(a,\mathrm{d}y )$ is then
\begin{equation*}
\begin{cases}
\cH^\PAR{p,q}(a)v(0,\mathrm{d}y )-v(a,\mathrm{d}y )=0 ,\\[0.05in]
-A(a)v(0,\mathrm{d}y )+{\pscale(0)\over \pqscale(a)}\cH^\PAR{p,q}(a)v(a,\mathrm{d}y) =\Um{p}{a}{a}.
\end{cases}
\end{equation*}
We deduce that $v(a,\mathrm{d}y )=\cH^\PAR{p,q}(a)v(0,\mathrm{d}y )$, $
v(0,\mathrm{d}y )=\Delta^{-1}\Um{p}{a}{a}$. 
Consequently, from the expression of $v(x,\mathrm{d}y )$ given by \eqref{eq:vsimple}, we finally have for $y>a$
\begin{equation*}
v(x,\mathrm{d}y )=\Delta^{-1}\Um{p}{a}{a}\dE_x\SBRA{ \mathrm{e}^{-p\tau_a^-}\cH^\PAR{p,q}(X_{\tau_a^-})}+\Um{p}{x}{a},
\end{equation*}
with $\Um{p}{.}{a}$ given by Equation \eqref{eq:Um} and $\dE_x\SBRA{ \mathrm{e}^{-p\tau_a^-}\cH^\PAR{p,q}(X_{\tau_a^-})}$ given by \eqref{eq:EH}. We finally have, for $y>a$,
\begin{multline*}
v(x,\mathrm{d}y )=-\pscale(x-y)\mathrm{d}y\\
+\Delta^{-1}\pscale(0) \mathrm{e}^{\Phi(p)(a-y)}\left( \mathcal{H}^\PAR{p,q}(x) - q \int_a^x \pscale(x-z) \mathcal{H}^\PAR{p,q}(z) \mathrm{d}z \right)\mathrm{d}y.
\end{multline*}

\bigskip
{\bf Case 2}: We now consider  $y\in[0,a]$, then the linear system satisfied by $v(0,\mathrm{d}y )$ and $v(a,\mathrm{d}y )$ becomes
\begin{equation*}
\begin{cases}
\cH^\PAR{p,q}(a)v(0,\mathrm{d}y )-v(a,\mathrm{d}y )={\pqscale(a)\over \pqscale(0)}\,C(0,\dy) ,\\[0.05in]
-A(a)v(0,\mathrm{d}y )+{\pscale(0)\over \pqscale(a)}\cH^\PAR{p,q}(a)v(a,\mathrm{d}y ) =C(a,\dy) .
\end{cases}
\end{equation*}
and its  unique solution is given by
\begin{align*}
v(0,\mathrm{d}y )&=\Delta^{-1}\PAR{\cH^\PAR{p,q}(a)C(0,\dy)+C(a,\dy)} ,\\[0.1in]
v(a,\mathrm{d}y )&=\Delta^{-1}\PAR{A(a){\pqscale(a)\over \pqscale(0)}C(0,\dy)+\cH^\PAR{p,q}(a)C(a,\dy)}.
\end{align*}

Let us now study $C(x,\dy)$. From Equation~\eqref{eq:U}, we get
$$
C(x,\dy) =\BRA{{\pqscale(a-y)\over \pqscale(a)}\dE_x\SBRA{ \mathrm{e}^{-p\tau_a^-}\pqscale(X_{\tau_a^-})}-\dE_x\SBRA{ \mathrm{e}^{-p\tau_a^-}\pqscale(X_{\tau_a^-}-y)}}\dy ,
$$
from which we deduce that
$$
C(0,\dy) = {\pqscale(0)\pqscale(a-y)\over \pqscale(a)}\dy .
$$
Introducing  the notation $B(.;.)$ and using Equation~\eqref{eq:Ba}, we get
\begin{align}\label{eq:C}
C(x,\dy)&=\BRA{{\pqscale(a-y)\over \pqscale(a)}B(x;a)-B(x-y;a-y)}\dy,\\
\label{eq:Ca}
C(a,\dy)&=\pscale(0)\PAR{\cH^\PAR{p,q}(a-y)-{\pqscale(a-y)\over \pqscale(a)}\cH^\PAR{p,q}(a)}\dy.
\end{align}

We deduce that
\[
v(x,\mathrm{d}y )=v(0,\mathrm{d}y)\dE_x\SBRA{ \mathrm{e}^{-p\tau_a^-}\cH^\PAR{p,q}(X_{\tau_a^-})} -  B(x-y;a-y) \mathrm{d}y ,
\]
with
$
v(0,\mathrm{d}y)=\Delta^{-1}\pscale(0)\cH^\PAR{p,q}(a-y) \mathrm{d}y .
$
So using Equations \eqref{eq:B} and \eqref{eq:EH}, we obtain for $y\in[0,a]$,
\begin{multline*}
v(x,\mathrm{d}y )=\left\{\Delta^{-1}\pscale(0)\cH^\PAR{p,q}(a-y)
\left( \mathcal{H}^\PAR{p,q}(x) - q \int_a^x \pscale(x-z) \mathcal{H}^\PAR{p,q}(z) \mathrm{d}z \right)\right.\\
\left.-
\pscale(x-y)-q\int_0^{a-y}\pscale(x-y-z)\pqscale(z)\mathrm{d}z \right\}\dy.
\end{multline*}

\bigskip
{\bf Case 3}: We finally consider the case where $y<0$. Then the linear system gives
\begin{equation*}
\begin{cases}
\cH^\PAR{p,q}(a)v(0,\mathrm{d}y )-v(a,\mathrm{d}y )={\pqscale(a)\over\pqscale(0)}D(0,\mathrm{d}y)\\[0.05in]
-A(a)v(0,\mathrm{d}y )+{\pscale(0)\over \pqscale(a)}\cH^\PAR{p,q}(a)v(a,\mathrm{d}y ) =D(a,\mathrm{d}y),
\end{cases}
\end{equation*}
and its unique solution is
\begin{align*}
v(0,\mathrm{d}y )&=\Delta^{-1}\PAR{\cH^\PAR{p,q}(a)D(0,\mathrm{d}y)+D(a,\mathrm{d}y)}\\
v(a,\mathrm{d}y )&=\Delta^{-1}\PAR{A(a){\pqscale(a)\over  \pqscale(0)}D(0,\mathrm{d}y)+\cH^\PAR{p,q}(a)D(a,\mathrm{d}y)}.
\end{align*}

Let us now study $D(x,\dy)$. 

From \eqref{eq:Up}, for $y<0$, we have  $\Up{p}{x}{0}=\BRA{ \mathrm{e}^{\Phi(p)x}\pscale(-y)-\pscale(x-y)}\dy $ for $x\leq a$ and then
\begin{multline*}
D(x,\dy)=\dE_x\SBRA{ \mathrm{e}^{-p\tau_a^-}\dE_{X_{\tau_a^-}}\SBRA{ \mathrm{e}^{-(p+q)\tau_0^-}\Up{p}{X_{\tau_0^-}}{0};\tau_0^-< \tau_a^+}}\\
=\BRA{\pscale(-y)A(x)-\dE_x\SBRA{ \mathrm{e}^{-p\tau_a^-}\dE_{X_{\tau_a^-}}\SBRA{ \mathrm{e}^{-(p+q)\tau_0^-}\pscale(X_{\tau_0^-}-y);\tau_0^-< \tau_a^+}}}\dy.
\end{multline*}

Using  Lemma \ref{lem:techlemma2} and $B(x-y;a-y)=\e_x \SBRA{\mathrm{e}^{-p\tau_a^-} \pqscale(X_{\tau_a^-} -y)}$, with an explicit form given by \eqref{eq:B}, we get
\begin{multline*}
\dE_x\SBRA{ \mathrm{e}^{-p\tau_a^-}\dE_{X_{\tau_a^-}}\SBRA{ \mathrm{e}^{-(p+q)\tau_0^-}\pscale(X_{\tau_0^-}-y);\tau_0^-< \tau_a^+}}\\
=B(x-y;a-y)-q\int_{0}^{-y}B(x-y-z;a-y-z)\pscale(z)\mathrm{d}z\\
-{B(x;a)\over \pqscale(a)}\PAR{\pqscale(a-y)-q\int_{0}^{-y}\pqscale(a-y-z)\pscale(z)\mathrm{d}z}.
\end{multline*}

From Equation \eqref{eq:A}, we deduce
\begin{multline*}
D(x,\dy)=\\
\left\{\pscale(-y)\dE_x\SBRA{ \mathrm{e}^{-p\tau_a^-}\cH^\PAR{p,q}(X_{\tau_a^-})}
-B(x-y;a-y)+q\int_{0}^{-y}B(x-y-z;a-y-z)\pscale(z)\mathrm{d}z\right.\\
\left.+{B(x;a)\over \pqscale(a)}\PAR{\pqscale(a-y)-q\int_{0}^{-y}\pqscale(a-y-z)\pscale(z)\mathrm{d}z-\pscale(-y)\cH^\PAR{p,q}(a)}\right\}\dy
\end{multline*}

We easily compute $
B(h;a+h)=\pqscale(h)$ for $h\geq 0$ and then, using Equations \eqref{eq:sym}, 
we have
\begin{multline*}
D(0,\dy)=\\
{\pqscale(0)\over\pqscale(a)}\BRA{\pqscale(a-y)-q\int_0^{-y}\pqscale(a-y-z)\pscale(z)\mathrm{d}z-\pscale(-y)\cH^\PAR{p,q}(a)}
\dy.
\end{multline*}

Consequently, we have
\begin{multline*}
D(x,\dy)-{B(x;a)\over \pqscale(0)}D(0,\mathrm{d}y)=\\
\BRA{\pscale(-y)\dE_x\SBRA{ \mathrm{e}^{-p\tau_a^-}\cH^\PAR{p,q}(X_{\tau_a^-})}
-B(x-y;a-y)
+q\int_0^{-y}B(x-y-z;a-y-z)\pscale(z)\mathrm{d}z}
\dy.
\end{multline*}
and then, from Equation \eqref{eq:vsimple} satisfied by $v(x,\dy)$, we deduce that
\begin{multline*}
v(x,\mathrm{d}y )=\PAR{v(0,\dy)+p\pscale(-y)\dy}\dE_x\SBRA{ \mathrm{e}^{-p\tau_a^-}\cH^\PAR{p,q}(X_{\tau_a^-})} \\
 -\BRA{B(x-y;a-y)
-q\int_0^{-y}B(x-y-z;a-y-z)\pscale(z)\mathrm{d}z}\dy,
\end{multline*}
Taking $x=a$ in \eqref{eq:EH}, we note that $\dE_a\SBRA{ \mathrm{e}^{-p\tau_a^-}\cH^\PAR{p,q}(X_{\tau_a^-})} =\cH^\PAR{p,q}(a)-\Delta$.
Then, using the expression of the solution of the linear system, the expression of $D(.,\dy)$ and Equation \eqref{eq:Ba}, we have
\begin{multline*}
v(0,\dy)+\pscale(-y)\dy=\Delta^{-1}\PAR{\cH^\PAR{p,q}(a)D(0,\mathrm{d}y)+D(a,\mathrm{d}y)+\Delta\pscale(-y)\dy}\\
=\Delta^{-1}\pscale(0)\BRA{\cH^\PAR{p,q}(a-y)-q\int_0^{-y}\cH^\PAR{p,q}(a-y-z)\pscale(z)\mathrm{d}z}\dy.
\end{multline*}
Introducing this value in the last expression of $v(x,\dy)$ and using \eqref{eq:EH}, we have
\begin{multline*}
v(x,\mathrm{d}y )=\left\{
 -\PAR{\pscale(x-a)\cH^\PAR{p,q}(a-y)+B(x-y;a-y)}\right.
\\ +q\int_0^{-y}\PAR{\pscale(x-a)\cH^\PAR{p,q}(a-y-z)+B(x-y-z;a-y-z)}\pscale(z)\mathrm{d}z  
\\
+\Delta^{-1}\pscale(0)\PAR{\cH^\PAR{p,q}(a-y)-q\int_0^{-y}\cH^\PAR{p,q}(a-y-z)\pscale(z)\mathrm{d}z}\\
\left.\times \PAR{\mathcal{H}^\PAR{p,q}(x) - q \int_a^x \pscale(x-z) \mathcal{H}^\PAR{p,q}(z) \mathrm{d}z }
\right\}\dy.
\end{multline*}
From \eqref{eq:B},
\[
\pscale(x-a)\cH^\PAR{p,q}(a)+B(x;a)=
\pscale(x)+q\int_{0}^{a} \pscale(x-u)\pqscale(u)\mathrm{d}u,
\]
then we finally deduce, for $y<0$,
\begin{multline*}
v(x,\mathrm{d}y )=\left\{-\pscale(x-y)-q\int_{0}^{a-y} \pscale(x-y-u)\pqscale(u)\mathrm{d}u\right.\\
+q\int_0^{-y}\PAR{\pscale(x-y-z)+q\int_{0}^{a-y-z} \pscale(x-y-z-u)\pqscale(u)\mathrm{d}u}\pscale(z)\mathrm{d}z\\
+\Delta^{-1}\pscale(0)\PAR{\cH^\PAR{p,q}(a-y)
-q\int_0^{-y}\pscale(z)\cH^\PAR{p,q}(a-y-z)\mathrm{d}z} \\
\left.\times\left( \mathcal{H}^\PAR{p,q}(x) - q \int_a^x \pscale(x-z) \mathcal{H}^\PAR{p,q}(z) \mathrm{d}z \right)\right\}\dy.
\end{multline*}

Introducing the value of $\Delta$ given by \eqref{eq:delta} and noting that the function $\mathcal{W}^\PAR{p,q}$ satisfies the relations:
$$
\mathcal{W}^{(p,q)}_{x-a} (x-y) = \mathcal{W}^{(p+q,-q)}_{a-y} (x-y) ,
$$
and, if $y>a$ then it is also equal to $\pscale(x-y)$, Theorem \ref{main} is proved for any $y\in\dR$ when $X$ is a process of BV.

\subsection{Proof of Theorem~\ref{main} when $X$ is of UBV}
 
We follow the argument of \cite{loeffenetal2014} to extend the result to SNLP with paths of unbounded variation (UBV).

The proof uses an approximation argument for which we need to introduce a sequence $\PAR{X^n}_{n\geq 1}$ of spectrally negative Lévy processes of bounded variation. To this end, suppose $X$ is a spectrally negative Lévy process having paths of unbounded variation, with Lévy triplet $(\gamma,\sigma,\Pi)$. Set, for each $n\geq 1$, the spectrally negative Lévy process $X^n= \BRA{X^n_t,t\geq 0}$ with Lévy triplet  $(\gamma,0,\Pi^n)$, where
\[
\Pi^n(\mathrm{d}\theta):=\ind_\BRA{\theta\geq 1/n}\Pi(\mathrm{d}\theta)+\sigma^2n^2\delta_{1/n}(\mathrm{d}\theta),
\]
with $\delta_{1/n}(\mathrm{d}\theta)$ standing for the Dirac point mass at $1/n$. Note that $X^n$ has paths of bounded variation with drift $\mathtt{c}^n:=\gamma+\int_{1/n}^1\theta\Pi(\mathrm{d}\theta)+\sigma^2n^2$, which means that $\mathtt{c}^n$ may be negative for small $n$. Note that $X^n$ is a \textit{true} spectrally negative Lévy process when $n$ is large enough $n$. By Bertoin \cite[p.210]{bertoin1996} , $X^n$ converges almost surely to $X$ uniformly on compact time intervals. We denote by $\pscale_n$ the $p-$scale function corresponding to the spectrally negative Lévy process $X^n$. We also introduce $\psi_n$ the Laplace exponent and $\Phi_n$ its right-inverse of $X^n$, with the convention $\inf\emptyset=\infty$.

From the result for bounded variation processes, we have for $y\geq 0$
\begin{multline}\label{eq:vn}
 \int_0^{\infty}\mathrm{e}^{-pt}\e_x \left[ \mathrm{e}^{- q \int_0^{t} \ind_{(0,a)} (X^n_s) \mathrm{d}s } ; X^n_{t} \in \mathrm{d}y \right]\mathrm{d}t=\\
\hskip -4cm  \mathrm{e}^{-\Phi_n(p)a}
 \left( \frac{\mathcal{H}_n^\PAR{p,q}(x) - q \int_a^x \pscale_n(x-z) \mathcal{H}_n^\PAR{p,q}(z) \mathrm{d}z}{\psi'_n(\Phi_n(p)) + q \int_0^a \mathrm{e}^{-\Phi_n(p) z} \mathcal{H}_n^\PAR{p,q}(z) \mathrm{d}z} \right)\\
 \times \left\lbrace \cH_n^\PAR{p,q}(a-y) - q \int_0^{-y} \cH_n^\PAR{p,q}(a-y-z) \pscale_n(z) \mathrm{d}z \right\rbrace \dy \\
- \left\{ \mathcal W_{n,x-a}^{(p,q)} (x-y) -q \int_0^{-y} \mathcal W_{n,x-a}^{(p,q)} (x-y-z) \pscale_n(z) \mathrm{d}z \right\} \dy .
\end{multline}
with  $
\displaystyle{\mathcal{W}_{n,a}^{(p,q)} (x) :=  W_n^{(p+q)}(x) -q \int_0^a W_n^{(p+q)}(x-z) \pscale_n(z) \mathrm{d}z }$ and
\[
\mathcal{H}_n^\PAR{p,q}(x) := \mathrm{e}^{\Phi_n(p) x} \left[ 1 + q \int_0^x \mathrm{e}^{-\Phi_n(p)z} W_n^{(p+q)}(z) \mathrm{d}z \right].
\]
Our aim in to take $n\rightarrow \infty$ in both sides of \eqref{eq:vn}.

By Bertoin \cite[p.210]{bertoin1996} , $X^n$ converges almost surely to $X$ uniformly on compact time intervals, i.e. for all $t>0$ $\lim_{n\rightarrow \infty}\sup_{s\in[0,t]}|X_s^n-X_s|=0$ $\dP_x-$a.s. 
Therefore $ \mathrm{e}^{- q \int_0^{t} \ind_{(0,a)} (X^n_s) \mathrm{d}s } \ind_{ X^n_{t} \in A} $ converges to  $ \mathrm{e}^{- q \int_0^{t} \ind_{(0,a)} (X_s) \mathrm{d}s } \ind_{ X_{t} \in A} $ a.s., for any Borel set $A$. Since $ \mathrm{e}^{- q \int_0^{t} \ind_{(0,a)} (X^n_s) \mathrm{d}s } \ind_{ X^n_{t} \in A} $ is dominated by $1$, it is easy to show that
the left hand side of  \eqref{eq:vn} converges to the desired expression.

\medskip
Let us now study the convergence of the right hand side of \eqref{eq:vn} when $n\rightarrow\infty$. The Laplace exponent $\psi_n$ of $X^n$ converges to the Laplace exponent $\psi$ of $X$, which implies $\Phi_n$ and $\psi'_n$ converge to $\Phi$ and $\psi'$ respectively.
It also means  via \eqref{def_scale} that the Laplace transform of $\pscale_n$ converges to the Laplace transform
of $\pscale$, and thanks to the continuity theorem of Laplace transforms, $\pscale_n(x)\rightarrow \pscale(x)$ for all $x\in\dR$ and $p\geq 0$. At last, since all the functions involved are continuous and since we consider compact sets, using the dominated convergence theorem on the definition of $\mathcal{W}_{n,a}^{(p,q)}$ and $\mathcal{H}_n^\PAR{p,q}$, we deduce the convergence of $\mathcal{W}_{n,a}^{(p,q)}$ and $\mathcal{H}_n^\PAR{p,q}$  to  $\mathcal{W}_{a}^{(p,q)}$ and $\mathcal{H}^\PAR{p,q}$ respectively when $n\rightarrow\infty$.

\end{appendices}

%
%
\bibliographystyle{abbrv}
\bibliography{occupation}

\end{document}